\documentclass[11pt, a4paper]{amsart}
\usepackage[utf8]{inputenc}
\usepackage{amssymb}
\usepackage{amsmath}
\usepackage{amsthm}
\usepackage{todonotes}
\usepackage{comment}
\usepackage{hyperref}
\date{}

\DeclareMathOperator{\St}{St}
\DeclareMathOperator{\Aut}{Aut}

\DeclareMathOperator{\Sym}{Sym}
\DeclareMathOperator{\Alt}{Alt}
\DeclareMathOperator{\Epi}{Epi}
\DeclareMathOperator{\Hom}{Hom}

\newcommand{\RR}{\mathbb{R}}
\newcommand{\NN}{\mathbb{N}}
\newcommand{\ZZ}{\mathbb{Z}}
\newcommand{\II}{\mathcal{I}}

\newtheorem{thm}{Theorem}[section]
\newtheorem{defn}[thm]{Definition}
\newtheorem{prop}[thm]{Proposition}

\newtheorem{ex}[thm]{Example}
\newtheorem{lemma}[thm]{Lemma}
\newtheorem{notation}[thm]{Notation}
\newtheorem{remark}[thm]{Remark}

\title[Subexponential growth of groups acting on rooted trees]{On the subexponential growth of groups acting on rooted trees}

\date{\today}
\author{Dominik Francoeur}
\begin{document}

\begin{abstract}
We show that every group in a large family of (not necessarily torsion) spinal groups acting on the ternary rooted tree is of subexponential growth.
\end{abstract}

\maketitle

\section{Introduction}

Word growth is an important quasi-isometry invariant of finitely generated groups. Such groups can be broadly divided into three distinct classes: groups of polynomial growth, groups of exponential growth and groups of intermediate growth (that is, groups whose growth is faster than any polynomial but slower than exponential).

While the existence of groups of polynomial or exponential growth is readily established, the question is far less obvious in the case of groups of intermediate growth and was first raised by Milnor in 1968 (\cite{Milnor68}). It was settled by Grigorchuk in 1983 \cite{Grigorchuk83,Grigorchuk85} when he proved that the infinite, finitely generated torsion group defined in \cite{Grigorchuk80}, now known as the first Grigorchuk group, is of intermediate growth.

Since this initial discovery, many other groups of intermediate growth have been found. One of those, now known as the Fabrykowski-Gupta group, was first studied by Fabrykowski and Gupta in \cite{FabrykowskiGupta85} and in \cite{FabrykowskiGupta91}. It is a self-similar branch group acting on the ternary rooted tree. The Fabrykowski-Gupta group was later revisited by Bartholdi and Pochon in \cite{BartholdiPochon09}, where they provided a different proof of the fact that its growth is intermediate, along with bounds on the growth.

In order to establish their results, Bartholdi and Pochon proved that under suitable conditions on the length of words, the growth of a self-similar group acting on a rooted tree is subexponential if and only if the growth of a subset of elements (those for which the projection to a certain fixed level does not reduce the length) is subexponential.

It turns out that the same holds if we consider the subset of elements whose length is never reduced by the projection to any level. As this set is smaller than the one considered by Bartholdi and Pochon, it can potentially be easier to show that its growth is subexponential.

After reviewing some basic results and definitions in Section \ref{sec:Preliminaries}, we will define non-$\ell_1$-expanding similar families of groups of rooted tree automorphisms and prove a generalized Bartholdi-Pochon criterion for such families of groups in Section \ref{sec:CompressionAndGrowth}. We will then use it in Section \ref{sec:GrowthSpinalGroups} to prove that a large family of spinal groups acting on the ternary rooted tree are of subexponential growth.


\section{Preliminaries}\label{sec:Preliminaries}

\subsection{Equivalence classes of non-decreasing functions}

Given two non-decreasing functions $f,g \colon \NN \rightarrow \NN$, we write $f\lesssim g$ if there exists $C\in \NN^*$ such that $f(n)\leq g(Cn)$ for all $n\in \NN^*$. The functions $f$ and $g$ are said to be equivalent, written $f\sim g$, if $f\lesssim g$ and $g\lesssim f$. We say that
\begin{itemize}
\item $f$ is of \emph{polynomial growth} if there exists $d\in \NN$ such that $f\lesssim n^d$
\item $f$ is of \emph{superpolynomial growth} if $n^d \lnsim f$ for all $d\in \NN$
\item $f$ is of \emph{exponential growth} if $f\sim e^n$
\item $f$ is of \emph{subexponential growth} if $f\lnsim e^n$
\item $f$ is of \emph{intermediate growth} if $f$ is of superpolynomial growth and of subexponential growth.
\end{itemize}

\subsection{Word pseudometrics and word growth}

Given a finite symmetric generating set of a group $G$, one can construct a natural metric on $G$ called the \emph{word metric}. For our purposes, it will be more convenient to consider the more general notion of a \emph{word pseudometric}.

\begin{defn}
Let $G$ be a finitely generated group and $S$ be a symmetric finite set generating $G$. A map $|\cdot|\colon S\rightarrow \{0,1\}$ that associates to every generator a length of $0$ or $1$ will be called a \emph{pseudolength} on $S$. A pseudolength can be extended to a map
\begin{align*}
|\cdot|\colon G &\rightarrow \NN \\
g &\mapsto \min_{g=s_1\dots s_k, s_i\in S}\left\{\sum_{i=1}^{k}|s_i| \right\}
\end{align*}
called the \emph{word pseudonorm} of $G$ (associated to $(S, |\cdot|)$). The corresponding pseudometric
\begin{align*}
d\colon G\times G &\rightarrow \NN\\
(g,h) &\mapsto |g^{-1}h|
\end{align*}
is the \emph{word pseudometric} of $G$ (associated to $(S, |\cdot|)$).
\end{defn}

\begin{remark}
If every generator is assigned a length of $1$, then the word pseudometric is in fact a metric, called the \emph{word metric}.
\end{remark}

If there is only a finite number of elements with length $0$, one can define a growth function for the group with regards to the given pseudometric. The growth function thus obtained is in fact equivalent to the usual growth function.

\begin{prop}
Let $G$ be a group generated by a finite symmetric set $S$ and $|\cdot| \colon S \rightarrow \{0,1\}$ be a pseudolength on $S$. If the subgroup
\[G_0 = \langle \{s\in S \mid |s|=0\}\rangle\]
is finite, then the growth function
\begin{align*}
\gamma_{G,S,|\cdot|} \colon \NN &\rightarrow \NN \\
n&\mapsto \left|B_{G,S,|\cdot|}(n)\right|
\end{align*}
is well-defined, where $B_{G,S,|\cdot|}(n)=\left\{g\in G \mid |g|\leq n\right\}$. Furthermore, $\gamma_{G,S,|\cdot|} \sim \gamma_{G,S}$, where $\gamma_{G,S}$ is the usual growth function obtained by giving length $1$ to each generator.
\end{prop}
\begin{proof}
To show that $\gamma_{G,S,|\cdot|}$ is well-defined, we need to show that
\[\left|B_{G,S,|\cdot|}(n)\right| < \infty\]
for every $n\in \NN$. For $g\in G$ with $|g|=n$, it follows from the definition of the word pseudonorm that there exist $s_1,\dots, s_n \in S_1$, $g_0,\dots, g_{n} \in G_0$ such that
\[g=g_0s_1g_1\dots s_n g_{n},\]
where $S_1 = \{s\in S \mid |s|=1\}$. Hence,
\[\left|B_{G,S,|\cdot|}(n)\right| \leq |G_0|^{n+1}|S_1|^{n} < \infty.\]

We must now show that the growth function with respect to the word pseudometric is equivalent to the growth function with the word metric. Let us denote by $|\cdot|_w$ the word metric in $G$ with respect to $S$. Let
\[M = \max\{|g|_w \mid g\in G_0\}.\]
Then, the decomposition
\[g=g_0s_1g_1\dots s_n g_{n}\]
implies that
\begin{align*}
|g|_w &\leq (n+1)M + n \\
&= |g|(M+1) + M \\
&\leq (2M+1)|g|
\end{align*}
if $|g|\geq 1$. Hence, if $n\geq 1$,
\[\gamma_{G,S,|\cdot|}(n) \leq \gamma_{G,S}((2M+1)n)\]
so $\gamma_{G,S,|\cdot|} \lesssim \gamma_{G,S}$. Since it is clear from the definition that $\gamma_{G,S} \lesssim \gamma_{G,S,|\cdot|}$, we have $\gamma_{G,S,|\cdot|} \sim \gamma_{G,S}$.
\end{proof}

\begin{remark}
A word pseudometric yielding a finite subgroup of length $0$ will be called a \emph{proper} word pseudometric. Since the growth function coming from a proper word pseudometric is equivalent to the growth function coming from a word metric, we will make no distinction between the two.
\end{remark}

In what follows, we will be interested mainly in distinguishing between groups of exponential or subexponential growth. For this purpose, it will be convenient to study a quantity called the exponential growth rate of the group.

\begin{prop}\label{prop:GrowthRateIsWellDefined}
Let $G$ be a finitely generated group, $S$ be a finite symmetric generating set and $|\cdot|\colon G \rightarrow \NN$ be a proper word pseudonorm. The limit
\[\kappa_{G,S, |\cdot|} = \lim_{n\to \infty}\gamma_{G,S, |\cdot|}(n)^{\frac{1}{n}}\]
exists and is called the \emph{exponential growth rate} of the group $G$ (with respect to the generating set $S$ and the pseudonorm $|\cdot|$).
\end{prop}

\begin{prop}
Let $G$ be a finitely generated group with a finite symmetric generating set $S$ and a proper word pseudonorm $|\cdot|$. Then, $\kappa_{G,S,|\cdot|} > 1$ if and only if $G$ is of exponential growth.
\end{prop}

It will sometimes be more convenient to consider spheres instead of balls. In the case of infinite finitely generated groups, the exponential growth rate can also be calculated from the size of spheres.

\begin{prop}\label{prop:SphericalGrowthRateIsTheSame}
Let $G$ be an infinite finitely generated group with finite symmetric generating set $S$ and proper word pseudonorm $|\cdot|$. Then,
\[\kappa_{G,S,|\cdot|} = \lim_{n\to\infty}|\Omega_{G,S,|\cdot|}(n)|^{\frac{1}{n}},\]
where $\Omega_{G,S,|\cdot|}(n)$ is the sphere of radius $n$ in the word pseudometric $|\cdot|$.
\end{prop}

\subsection{Rooted trees}

Let $d>1$ be a natural number and $T_d$ be the $d$-regular rooted tree. The set of vertices of $T_d$ is $V(T_d) = X^*$, the set of finite words in the alphabet $X=\{1,2,\dots d\}$. We will often abuse the notation and write $v\in T_d$ instead of $v\in V(T_d)$ or $v\in X^*$ to refer to a vertex of $T_d$. The set of edges is
\[E(T_d)= \left\{\{w,wx\} \subset X^* \mid w\in X^*, x\in X\right\}.\]
A vertex $w'\in X^*$ is said to be a \emph{child} of $w\in X^*$ if $w'=wx$ for some $x\in X$. In this case, $w$ is called the \emph{parent} of $w'$. For $n\in \NN$, the set $L_n\subset X^*$ of words of length $n$ is called the $n$-th level of the tree.

The group of automorphisms of $T_d$, that is, the group of bijections of $X^*$ that leave $E(T_d)$ invariant, will be denoted by $\Aut(T_d)$. For $G\leq \Aut(T_d)$ and $v\in T_d$, we will write $\St_G(v)$ to refer to the stabilizer of $v$ in $G$. For $n\in \NN$, we define the stabilizer of the $n$-th level in $G$, denoted $\St_G(n)$, as
\[\St_G(n) = \bigcap_{v\in L_n}\St_G(v).\]
When $G=\Aut(T_d)$, we will often simply write $\St(v)$ and $\St(n)$ in order to make the notation less cluttered.

For $v\in T_d$, we will denote by $T_v$ the subtree rooted at $v$. This subtree is naturally isomorphic to $T_d$. Any $g\in St(v)$ leaves $T_v$ invariant. The restriction of $g$ to $T_v$ is therefore (under the natural isomorphism between $T_v$ and $T_d$) an automorphism of $T_d$, which we will denote by $g|_v$. The map
\begin{align*}
\varphi_v\colon \St(v) &\rightarrow \Aut(T_d) \\
g&\mapsto g|_v
\end{align*}
is clearly a homomorphism. This allows us to define a homomorphism
\begin{align*}
\psi\colon \St(1)&\rightarrow \Aut(T_d)^d \\
g&\mapsto \left(g|_1,g|_2,\dots, g|_d\right).
\end{align*}
We can also define homomorphisms
\[\psi_n\colon \St(n) \rightarrow \Aut(T_d)^{d^n}\]
for all $n\in \NN^*$ inductively by setting $\psi_1 = \psi$ and 
\[\psi_{n+1}(g) = (\psi_n(g|_1), \psi_n(g|_2), \dots, \psi_n(g|_d)).\]
It is clear from the definition that these homomorphisms are in fact isomorphisms. In what follows, we will use the same notation for those maps and their restriction to some subgroup $G\leq \Aut(T_d)$. In this case, it is important to note that the maps $\psi_n$ are still injective, but are no longer surjective in general.

There is a faithful action of $\Sym(d)$ on $T_d$ given by
\[\tau(xw) = \tau(x)w\]
where $\tau\in \Sym(d)$, $x\in X$ and $w\in X^*$. This action gives us an embedding of $\Sym(d)$ in $T_d$ and in what follows, we will often identify $\Sym(d)$ with its image in $T_d$ under this embedding. The automorphisms in the image of $\Sym(d)$ are called \emph{rooted automorphisms}. Any $g\in \Aut(T_d)$ can be uniquely written as $g=h\tau$ with $h\in \St(1)$ and $\tau\in \Sym(d)$. In a slight abuse of notation, we will often find it more convenient to write
\[g=(g_1,g_2,\dots, g_d)\tau,\]
where $(g_1,g_2,\dots,g_d) = \psi(h)$.

More generally, for any $n\in \NN^*$, there is a natural embedding of the wreath product $\Sym(d)\wr \dots \wr\Sym(d)$ of $\Sym(d)$ with itself $n$ times in $\Aut(T_d)$ and any $g\in \Aut(T_d)$ can be written uniquely as $g=h\tau$ with $h\in \St(n)$ and $\tau \in \Sym(d)\wr\dots\wr\Sym(d)$. In the same fashion as above, we will often write
\[g= (g_{11\dots 1}, g_{11\dots 2}, \dots, g_{dd\dots d})\tau\]
where $(g_{11\dots 1}, g_{11\dots 2}, \dots, g_{dd\dots d}) = \psi_n(h)$.

\section{Incompressible elements and growth}\label{sec:CompressionAndGrowth}

\subsection{Non-$\ell_1$-expanding similar families of groups acting on rooted trees}

A classical way of showing that a group acting on a rooted tree is of subexponential growth is to show that the projection of elements to some level induces a significant amount of length reduction. We introduce here a class of groups that seem well suited to this kind of argument, \emph{non-$\ell_1$-expanding similar families} of groups acting on rooted trees. This is a restriction of the more general notion of \emph{similar families} of groups as defined by Bartholdi in \cite{Bartholdi15}.

Note that for the rest of this section, $d$ will denote an integer greater than $1$.

\begin{defn}
Let $\Omega$ be a set and $\sigma\colon \Omega \rightarrow \Omega$ be a map from this set to itself. For each $\omega\in \Omega$, let $G_\omega \leq \Aut(T_d)$ be a group of automorphisms of $T_d$ acting transitively on each level, generated by a finite symmetric set $S_\omega$ and endowed with a proper word pseudonorm $|\cdot|_\omega$. The family $\{(G_\omega, S_\omega, |\cdot|_\omega)\}_{\omega\in \Omega}$ is a \emph{similar family} of groups of automorphisms of $T_d$ if for all $\omega\in \Omega$ and all $g\in G_\omega$,
\[g=(g_1,g_2, \dots, g_d)\tau\]
with $g_1,g_2,\dots,g_d \in G_{\sigma(\omega)}$, $\tau\in \Sym(d)$. Furthermore, if
\[\sum_{i=1}^d |g_i|_{\sigma(\omega)} \leq |g|_\omega,\]
the similar family $\{(G_\omega, S_\omega, |\cdot|_\omega)\}_{\omega\in \Omega}$ is said to be a \emph{non-$\ell_1$-expanding similar family} of groups of automorphisms of $T_d$.

In the case where $|\Omega|=1$, a similar family contains only one group, which is then said to be \emph{self-similar}.
\end{defn}

\begin{remark}
There is a more general notion of a similar family of groups in which the groups need not act on regular rooted trees, but only on spherically homogeneous rooted trees. However, we will not need such generality for what follows.
\end{remark}

\begin{remark}
In what follows, we will frequently consider only (non-$\ell_1$-expanding) similar families of the form $\{(G_\nu, S_\nu, |\cdot|_\nu)\}_{\nu\in \NN}$, where the map from $\NN$ to $\NN$ is the addition by $1$. This has the advantage of simplifying the notation without causing any significant loss in generality. Indeed, let $\{(G_\omega, S_\omega, |\cdot|_\omega)\}_{\omega\in \Omega}$ be a (non-$\ell_1$-expanding) similar family. Then, for any $\omega\in \Omega$, $\{(G_\nu, S_\nu, |\cdot|_\nu)\}_{\nu\in \NN}$ is also a (non-$\ell_1$-expanding) similar family, where $(G_\nu, S_\nu, |\cdot|_\nu) = (G_{\sigma^{\nu}(\omega)}, S_{\sigma^{\nu}(\omega)}, |\cdot|_{\sigma^{\nu}(\omega)})$.
\end{remark}

\begin{remark}
If $\{(G_\nu, S_\nu, |\cdot|_\nu)\}_{\nu\in \NN}$ is a non-$\ell_1$-expanding similar family of groups, then for any $\nu\in \NN$ and $s=(s_1,s_2,\dots,s_d)\tau\in S_\nu$, we have
\[\sum_{i=1}^{d}{|s_i|_{\nu+1}} \leq |s|_{\nu} \leq 1\]
so there is at most one $s_i$ with positive length (and none if $|s|_{\nu}=0$).
\end{remark}

\begin{notation}
In order to keep the notation simple, if $\{(G_\nu, S_\nu, |\cdot|_\nu)\}_{\nu\in \NN}$ is a non-$\ell_1$-expanding similar family of groups, for $\nu\in \NN$, we will write $\gamma_\nu$ for the growth function and $\kappa_\nu$ for the exponential growth rate of $G_\nu$ with respect to the pseudonorm $|\cdot|_\nu$.
\end{notation}

The exponential growth rates of a non-$\ell_1$-expanding similar family of groups form a non-decreasing sequence.

\begin{prop}\label{prop:GrowthRateCannotDecrease}
Let $\{(G_\nu,S_\nu,|\cdot|_\nu)\}_{\nu\in \NN}$ be a non-$\ell_1$-expanding similar family of groups of automorphisms of $T_d$. For any $\nu\in \NN$,
\[\kappa_\nu \leq \kappa_{\nu+1}.\]
\end{prop}
\begin{proof}
Let $n\in \NN$ be greater than $d$ and let $g\in G_\nu$ be such that $|g|_\nu \leq n$. We have
\[g=(g_1,g_2,\dots,g_d)\tau\]
with $g_1,g_2,\dots, g_d\in G_{\nu+1}$, $\tau \in \Sym(d)$ and
\[\sum_{i=1}^d|g_i|_{\nu+1} \leq |g|_\nu = n.\]
Since $g$ is determined by $g_1,g_2,\dots, g_d$ and $\tau$, we have
\[\gamma_\nu(n) \leq d!\sum_{r_1+r_2+\dots + r_d \leq n}\gamma_{\nu+1}(r_1)\gamma_{\nu+1}(r_2)\dots \gamma_{\nu+1}(r_d).\]

Let $C(k) = \frac{\gamma_{\nu+1}(k)}{\kappa_{\nu+1}^{k}}$ for any $k\in \NN$.We have
\begin{align*}
\gamma_{\nu}(n) &\leq d!\sum_{r_1+r_2+\dots + r_d \leq n}C(r_1)\kappa_{\nu+1}^{r_1}C(r_2)\kappa_{\nu+1}^{r_2}\dots C(r_d)\kappa_{\nu+1}^{r_d} \\
&=d!\kappa_{\nu+1}^n \sum_{r_1+r_2+\dots + r_d \leq n}C(r_1)C(r_2)\dots C(r_d).
\end{align*}
Let $s(n)\in \{1,\dots, n\}$ be such that $C(s(n)) \geq C(r)$ for all $1\leq r \leq n$. We then have
\begin{align*}
\gamma_{\nu}(n) &\leq d!\kappa_{\nu+1}^n \sum_{r_1+r_2+\dots + r_d \leq n}C(s(n))^d \\
&\leq d! \kappa_{\nu+1}^n C(s(n))^d n^{d}
\end{align*}
It is clear from the definition that the sequence $s(n)$ is non-decreasing. Therefore, either it stabilizes or it goes to infinity. Since $\lim_{k\to\infty}C(k)^{\frac{1}{k}}=1$, in both cases we have $\lim_{n\to\infty}C(s(n))^{\frac{1}{n}}=1$. Hence,
\begin{align*}
\kappa_\nu &=\lim_{n\to\infty} \gamma_\nu(n)^\frac{1}{n} \\
&\leq \kappa_{\nu+1} \lim_{n\to\infty} d!^{\frac{1}{n}} C(s(n))^{\frac{1}{n}} n^{\frac{d}{n}} \\
&=\kappa_{\nu+1}.
\end{align*}
\end{proof}

\subsection{Examples}

Let us now present some examples of non-$\ell_1$-expanding similar families of groups of automorphisms of $T_d$.

\subsubsection{Spinal groups}\label{subsubsection:SpinalGroups}

Spinal groups were first introduced and studied, in a more restrictive version, by Bartholdi and \v{S}uni\'{k} in \cite{BartholdiSunic01}. A more general version was later introduced by Bartholdi, Grigorchuk and \v{S}uni\'{k} in \cite{BartholdiGrigorchukSunic03}. Spinal groups form a large family of groups which include many previously studied examples, such as the Grigorchuk groups and the Gupta-Sidki group. Note that the definition we give here is not the most general one, because we consider only regular rooted trees.

Let $B$ be a finite group and
\[\Omega = \left\{\{\omega_{ij}\}_{i\in\NN, 1\leq j < d} \mid \omega_{ij}\in\Hom(B,\Sym(d)), \bigcap_{i\geq k}\bigcap_{j=1}^{d-1} \ker\omega_{ij} = \{1\} \forall k \in \NN\right\}\]
be the set of sequences of homomorphisms from $B$ to $\Sym(d)$ such that the intersection of the kernels is trivial no matter how far into the sequence we start. Let
\begin{align*}
\sigma\colon \Omega&\rightarrow \Omega \\
\omega = \{\omega_{ij}\}_{i\in\NN, 1\leq j \leq d - 1} &\mapsto \sigma(\omega) = \{\omega_{(i+1)j}\}_{i\in\NN, 1\leq j \leq d - 1}
\end{align*}
be the \emph{left-shift} (with respect to the first index), which is well-defined thanks to the way the condition on the kernels was formulated.

For each $\omega = \{\omega_{ij}\}_{i\in\NN, j\in\{1,2,\dots,d-1\}} \in \Omega$, we can recursively define a homomorphism
\begin{align*}
\beta_\omega \colon G_B &\rightarrow \Aut(T_d) \\
b &\mapsto (\omega_{01}(b), \omega_{02}(b),\dots, \omega_{0(d-1)}(b), \beta_{\sigma(\omega)}(b))
\end{align*}
where, as usual, we identify $\Sym(d)$ with rooted automorphisms of $T_d$. The condition on the kernels of sequences in $\Omega$ ensures that this homomorphism is injective. Let us write $B_\omega = \beta_\omega(B)\leq \Aut(T_d)$.

For a fixed $\omega = \{\omega_{ij}\}\in \Omega$, let $A_\omega\leq \Sym(d)$ be any subgroup of $\Sym(d)$. For any $k\in\NN^*$, we then define
\[A_{\sigma^k(\omega)} = \left\langle\bigcup_{j=1}^{d}\omega_{kj}(B)\right\rangle.\]

\begin{defn}
Using the notation above, the group $G_\omega = \langle A_\omega, B_\omega \rangle$ for some $\omega\in \Omega$ and $A_\omega \leq \Sym(d)$ is a \emph{spinal group} if $A_{\sigma^k(\omega)}$ acts transitively on $\{1,2,\dots,d\}$ for all $k\in \NN$.
\end{defn}

\begin{remark}
For $\omega\in \Omega$ and $A_\omega\leq \Sym(d)$, if $G_\omega = \langle A_\omega, B_\omega \rangle$ is a spinal group, then $G_{\sigma^k(\omega)} = \langle A_{\sigma^k(\omega)}, B_{\sigma^k(\omega)}\rangle$ is a spinal group for all $k\in \NN$.
\end{remark}

For any spinal group $G_\omega$, the set $S_\omega = A_\omega \cup B_\omega$ is a finite symmetric generating set. Let $|\cdot|_\omega \colon S_\omega \rightarrow \{0,1\}$ be defined by
\[|g|_\omega = 
\begin{cases}
0 & \text{ if } g\in A_\omega \\
1 & \text{ otherwise.}
\end{cases}\]
It is clear from the definition that if $g\in S_\omega$, we have
\[g=(g_1,g_2,\dots,g_d)\tau\]
with $\tau\in A_\omega$, $g_1,g_2,\dots, g_d\in S_{\sigma(\omega)}$ and
\[\sum_{i=1}^d|g_i|_{\sigma(\omega)} = |g|_\omega.\]
As explained above, $|\cdot|_\omega$ can be extended to a word pseudonorm on $G_\omega$ that we will also denote by $|\cdot|_\omega$. The set of elements of length $0$ in this pseudonorm is exactly $A_\omega$, which is finite, and since it is true for the generators, we have that for $g\in G_\omega$,
\[g=(g_1,g_2,\dots,g_d)\tau\]
with $\tau\in A_\omega$, $g_1,g_2,\dots, g_d\in G_{\sigma(\omega)}$ and
\[\sum_{i=1}^d|g_i|_{\sigma(\omega)} \leq |g|_\omega.\]
Hence,
\[\{G_\omega, S_\omega, |\cdot|_\omega\}_{\omega\in\Omega}\]
is a non-$\ell_1$-expanding similar family of groups.

\begin{ex}[The first Grigorchuk group]
Let $d=2$, $A = \Sym(2)\cong \ZZ/2\ZZ$ and $B = (\ZZ/2\ZZ)^2$. Let $a$ be the non-trivial element of $A$ and $b,c,d$ be the non-trivial elements of $B$. For $x\in \{b,c,d\}$, let $\omega_{x} \colon B\rightarrow A$ be the epimorphism that sends $x$ to $1$ and the other two non-trivial letters to $a$. The group $G_\omega$ with $\omega = \omega_d\omega_c\omega_b\omega_d\omega_c\omega_b\dots$ (here, since $d-1 = 1$, there is only one index) and $A_\omega=A$ is the first Grigorchuk group, which was first introduced in \cite{Grigorchuk80}.
\end{ex}
\begin{ex}[Grigorchuk groups]
More generally, let $d=p$, where $p$ is a prime number, $A = \langle(1 2 \dots p)\rangle\cong \ZZ/p\ZZ$ and $B=(\ZZ/p\ZZ)^2$. Let
\begin{align*}
\phi_k\colon (\ZZ/p\ZZ)^2 &\rightarrow \ZZ/p\ZZ \\
(x,y) &\mapsto x+ky
\end{align*}
for $0\leq k \leq p-1$ and let
\begin{align*}
\phi_p\colon (\ZZ/p\ZZ)^2 &\rightarrow \ZZ/p\ZZ \\
(x,y) &\mapsto y.
\end{align*}
The groups $G_\omega$ with $A_\omega = A$ and $\omega=\{\omega_{ij}\}_{i\in\NN, 1\leq j \leq p-1} \in \Omega$ such that $\omega_{i1} =\phi_{k_i}$ for all $i\in\NN$ and $\omega_{ij} = 1$ if $j\ne 1$ are called Grigorchuk groups and were studied in \cite{Grigorchuk85}.
\end{ex}
\begin{ex}[\v{S}uni\'{k} groups]\label{ex:SunicGroups}
Let $d=p$ a prime number, $A = \langle(1 2 \dots p)\rangle\cong \ZZ/p\ZZ$ and $B=(\ZZ/p\ZZ)^m$ for some $m\in \NN$. Let $\phi\colon (\ZZ/p\ZZ)^{m}\rightarrow \ZZ/p\ZZ$ be the epimorphism given by the matrix
\[\begin{pmatrix}
0&0&\dots&0&1
\end{pmatrix}\]
in the standard basis, $\rho\colon (\ZZ/p\ZZ)^m \rightarrow (\ZZ/p\ZZ)^m$ be the automorphism given by
\[\begin{pmatrix}
0&0&\dots & 0 & -1 \\
1 & 0 & \dots & 0& a_1 \\
0&1&\dots&0& a_2 \\
\vdots & \vdots & \ddots & \vdots & \vdots \\
0&0&\dots&1 & a_{m-1}
\end{pmatrix}\]
in the standard basis, for some $a_1,\dots, a_{m-1}\in \ZZ/p\ZZ$, and $\omega=\{\omega_{ij}\} \in \Omega$, where
\[\omega_{i1} = \phi\circ \rho^i\]
and $\omega_{ij} = 1$ if $j\ne 1$. For any such $\omega$, we let $A_\omega = A$.

The groups $G_\omega$ that can be constructed in this way are exactly the groups that were introduced and studied by \v{S}uni\'{k} in \cite{Sunic07}.
\end{ex}

\begin{ex}[GGS groups]
GGS groups form another important family of examples of spinal groups. They are a generalization of the second Grigorchuk group (introduced in \cite{Grigorchuk80}) and the groups introduced by Gupta and Sidki in \cite{GuptaSidki83}. We present here the definition of GGS groups that was given in \cite{BartholdiGrigorchukSunic03}.

Let $A=\langle(1 2 \dots d)\rangle \cong \ZZ/d\ZZ$, $B=\ZZ/d\ZZ$ and $\epsilon = (\epsilon_1,\epsilon_2,\dots, \epsilon_{d-1}) \in (\ZZ/d\ZZ)^d$ such that $\epsilon \ne 0$. Let $\omega = \{\omega\}_{ij}\in \Omega$, where
\[\omega_{ij}(\bar{1}) = a^{\epsilon_j}\]
$\bar{1}\in B$ is the equivalence class of $1$ and $a=(1 2 d) \in A$, and let $A_\omega = A$. If $\gcd(\epsilon_1,\epsilon_2,\dots, \epsilon_{d-1},d) = 1$, then $G_\omega$ is called a GGS group.
\end{ex}

\subsubsection{Nekrashevych's family of groups $\mathcal{D}_\omega$}

Let $\{0,1\}^{\NN}$ be the set of infinite sequences of $0$ and $1$ and
\begin{align*}
\sigma\colon \{0,1\}^{\NN} &\rightarrow \{0,1\}^{\NN} \\
\omega_0\omega_1\omega_2\dots &\mapsto \omega_1\omega_2\omega_3\dots
\end{align*}
be the left-shift. For $\omega = \omega_0\omega_1\omega_2\dots\in \{0,1\}^{\NN}$, we can recursively define automorphisms $\beta_\omega, \gamma_\omega \in \Aut(T_2)$ by
\begin{align*}
\beta_\omega &= (\alpha, \gamma_{\sigma(\omega)})\\
\gamma_\omega &= 
\begin{cases}
(\beta_{\sigma(\omega)}, 1) & \text{ if } \omega_0=0 \\
(1, \beta_{\sigma(\omega)}) & \text{ if } \omega_0=1
\end{cases}
\end{align*}
where $\alpha\in \Aut(T_2)$ is the non-trivial rooted automorphism of $T_2$. We can then define the group $\mathcal{D}_\omega = \langle \alpha, \beta_\omega, \gamma_\omega \rangle$. This family of groups was first studied by Nekrashevych in \cite{Nekrashevych07}.

It follows from the definition that $\alpha^2=\beta_\omega^2 = \gamma_\omega^2 = 1$. Hence, the set $S_\omega = \{\alpha, \beta_\omega, \gamma_\omega\}$ is a finite symmetric generating set of $\mathcal{D}_\omega$. Let $|\cdot|_\omega \colon S\rightarrow \{0,1\}$ be given by $|\alpha|_\omega = 0$, $|\beta_\omega|_\omega = |\gamma_\omega|_\omega = 1$. Then, the family $\{(G_\nu, S_\nu, |\cdot|_\nu)\}_{\nu\in\NN}$ is a non-$\ell_1$-expanding similar family of automorphisms of $T_2$, where $G_\nu = \mathcal{D}_{\sigma^\nu(\omega)}$, $S_\nu = S_{\sigma^\nu(\omega)}$ and $|\cdot|_\nu = |\cdot|_{\sigma^\nu(\omega)}$.

\subsubsection{Peter Neumann's example}

We present here a group that first appeared as an example in Neumann's paper \cite{Neumann86}. The description we use here is based on \cite{BartholdiGrigorchukSunic03}.

Let $A=\Alt(6)$ and $X=\{1,2,\dots,6\}$. For every couple $(a,x)\in A\times X$ such that $x$ is a fixed point of $a$, we can recursively define an automorphism of $\Aut(T_6)$ by
\[b_{(a,x)} = (1,\dots, b_{(a,x)},\dots, 1)a\]
where the $b_{(a,x)}$ is in the x\textsuperscript{th} position. Let
\[S=\left\{b_{(a,x)} \in \Aut(T_6) \mid (a,x)\in A\times X \right\},\]
$G=\langle S \rangle$ and $|\cdot|\colon G \rightarrow \NN$ be the word norm associated to $S$. Then, it is clear from the definition that $\{(G_\nu,S_\nu,|\cdot|_\nu)\}_{\nu\in\NN}$ is a non-$\ell_1$-expanding similar family of automorphisms of $T_6$, where $G_\nu=G$, $S_\nu=S$ and $|\cdot|_\nu = |\cdot|$ for all $\nu\in \NN$. Hence, $G$ is a non-$\ell_1$-expanding self-similar group.

\subsection{Incompressible elements}

Let $\{(G_\nu, S_\nu, |\cdot|_\nu)\}_{\nu\in\NN}$ be a non-$\ell_1$-expanding similar family of groups of automorphisms of $T_d$. For any $k\in \NN^*$, we recursively define the sets $\II_k^{\nu}$ of elements of $G_\nu$ which have no length reduction up to level $k$ as
\[\II_k^{\nu} = \{g = (g_1,g_2,\dots, g_d)\tau \in G_\nu \mid g_1,g_2,\dots, g_d \in \II_{k-1}^{\nu+1}, \sum_{i=1}^d|g_i|_{\nu+1} = |g|_\nu\}\]
where $\II_0^\nu = G_\nu$ for all $\nu\in \NN$.

We will call the set
\[\II_\infty^\nu = \bigcap_{k=1}^\infty \II_k^\nu\]
the set of incompressible elements of $G_\nu$. This is the set of elements which have no length reduction on any level.

\subsection{Growth of incompressible elements}

We will see that if every group in a non-$\ell_1$-expanding similar family of groups of automorphisms of $T_d$ is generated by incompressible elements and the sets of incompressible elements grow uniformly subexponentially, then the groups themselves are also of subexponential growth. This result is a generalization of the first part of Proposition 5 in \cite{BartholdiPochon09}. The main difference is that we show here that under our assumptions, it is sufficient to look at the growth of the set $\II_\infty^\nu$ of incompressible elements instead of the set $\II_k^\nu$ of elements which have no reduction up to level $k$ for some $k\in \NN$.

\begin{thm}\label{thm:GrowthCriterion}
Let $A\in\NN$ be an integer, $\{(G_\nu, S_\nu, |\cdot|_\nu)\}_{\nu\in\NN}$ be a non-$\ell_1$-expanding similar family of automorphisms of $T_d$ such that $S_\nu \subseteq \II_\infty^\nu$ and $|S_\nu|\leq A$ for every $\nu\in\NN$, and let $\Omega_\nu(n)$ be the sphere of radius $n\in\NN$ in $G_\nu$ with respect to the pseudometric $|\cdot|_\nu$. If there exists a subexponential function $\delta\colon \NN\rightarrow\NN$ with $\ln(\delta)$ concave such that for infinitely many $\nu\in\NN$, $\II_\infty^\nu\cap\Omega_\nu(n) \leq \delta(n)$ for all $n\in \NN$, then the groups $G_\nu$ are of subexponential growth for every $\nu\in\NN$.
\end{thm}
\begin{proof}
The proof is inspired by the one found in \cite{BartholdiPochon09}, with a few key modifications. The idea is to split the set $\Omega_\nu(n)$ in two, the set of elements which can be written as a product of a few incompressible elements and the set of elements which can only be written as a product of a large number of incompressible elements. The first set grows slowly because there are few incompressible elements, and the second set grows slowly because there is a significant amount of length reduction.

Let us fix $\nu\in \NN$ such that $\II_\infty^\nu\cap\Omega_\nu(n) \leq \delta(n)$ for all $n\in \NN$. In what follows, we will show that $\kappa_\nu = 1$. By Proposition \ref{prop:GrowthRateCannotDecrease}, this will show that $\kappa_{\nu'} =1$ for all $\nu'\leq \nu$.

Since $S_\nu\subseteq \II_\infty^\nu$, we have that for every $g\in G_\nu$, the set
\[\left\{N\in\NN \mid g=g_1g_2\dots g_N, g_i\in \II_{\infty}^\nu, \sum_{i=1}^N|g_i|_\nu = |g|_\nu\right\}\]
is not empty. Hence, we can define
\[N(g) = \min\left\{N\in\NN \mid g=g_1g_2\dots g_N, g_i\in \II_{\infty}^\nu, \sum_{i=1}^N|g_i|_\nu = |g|_\nu\right\}.\]

For any $n\in \NN$ and $0<\epsilon < 1$, the sphere of radius $n$ in $G_\nu$, $\Omega_\nu(n)$, can be partitioned in two by the subsets
\begin{align*}
\Omega_{\nu}^{>}(n,\epsilon) &= \{g\in\Omega_\nu(n) \mid N(g) > \epsilon n\} \\
\Omega_{\nu}^{<}(n,\epsilon) &= \{g\in\Omega_\nu(n) \mid N(g) \leq \epsilon n\}.
\end{align*}

Let $g\in \Omega_\nu^{>}(n,\epsilon)$. By definition of $N(g)$, there exists $g_1,g_2,\dots, g_{N(g)}\in \II_\infty^{\nu}$ such that $g=g_1g_2\dots g_{N(g)}$ and $\sum_{i=1}^{N(g)}|g_i|_\nu = |g|_\nu$. Let $h_i = g_{2i-1}g_{2i}$ for $1\leq i \leq \lfloor\frac{N(g)}{2}\rfloor$. Then,
\[g=\begin{cases}h_1h_2\dots h_{\frac{N(g)-1}{2}}g_{N(g)} & \text{ if } N(g) \text{ is odd} \\
h_1h_2\dots h_{\frac{N(g)}{2}} & \text{ if } N(g) \text{ is even}. \end{cases}\]
Notice that since
\[|g|_\nu = \sum_{i=1}^{N(g)}|g_i|_\nu,\]
we must have $|h_i|_\nu = |g_{2i-1}|_\nu + |g_{2i}|_\nu $. Hence, no $h_i$ can be in $\II_\infty^\nu$ (otherwise, this would contradict the minimality of $N(g)$).

Let
\[S(g) = \left\{i \biggm| |h_i|_\nu \leq \frac{6}{\epsilon}\right\}\]
be the set of "small" factors of $g$ and
\[L(g) = \left\{i \biggm| |h_i|_\nu > \frac{6}{\epsilon}\right\}\]
be the set of "large" factors. Clearly, $|S(g)| + |L(g)| = \lfloor \frac{N(g)}{2}\rfloor$. Since $g\in \Omega_\nu^{>}(n,\epsilon)$, $N(g)$ is not too small compared to $n$, which implies that as long as $n$ is large enough, more than half of the factors of $g$ must be small. More precisely, if $n>\frac{3}{\epsilon}$, then $|S(g)| \geq \frac{1}{2}\lfloor \frac{N(g)}{2}\rfloor$. Indeed, it that were not the case, then we would have $|L(g)|> \lfloor \frac{N(g)}{2}\rfloor$, so
\begin{align*}
n &\geq \sum_{i=1}^{\lfloor \frac{N(g)}{2}}|h_i|_\nu \geq \sum_{h_i\in L(g)} |h_i|_\nu \\
&> \frac{1}{2}\left\lfloor\frac{N(g)}{2}\right\rfloor \frac{6}{\epsilon} \geq \frac{\left(N(g)-1\right)}{4}\frac{6}{\epsilon} \\
&> \frac{3}{2}n - \frac{3}{2\epsilon} \\
&> n
\end{align*}
which is a contradiction. Therefore, if $n>\frac{3}{\epsilon}$,
\begin{align*}
|S(g)| &\geq \frac{1}{2}\left\lfloor \frac{N(g)}{2} \right\rfloor \geq \frac{N(g)-1}{4} \\
&>\frac{\epsilon}{4}n - \frac{1}{4}\\
&> \frac{\epsilon}{8}n.
\end{align*}
This means that the number of small factors is comparable with $n$. This is important because, as we will see, every small factor gives us some length reduction on a fixed level (fixed in the sense that it does not depend on $n$, but only on $\epsilon$). Hence, on this level, we will see a large amount of length reduction.

For $r\in \RR$, let
\[l_\nu(r) = \max\{k\in\NN \mid \left(G_\nu \setminus \II_\infty^\nu\right)\cap B_\nu(r) \subseteq \II_k^\nu\} + 1\]
where $B_\nu(r)$ is the ball of radius $r$ in $G_\nu$. Notice that since $\left(G_\nu \setminus \II_\infty^\nu\right)\cap B_\nu(r)$ is finite and not contained in $\II_\infty^\nu$, $l_\nu(r)$ is well-defined (i.e. finite).

Let us consider the $l_\nu(\frac{6}{\epsilon})\textsuperscript{th}$-level decomposition of $g$,
\[g=(g_{11\dots 1}, g_{11\dots 2}, \dots, g_{dd\dots d})\tau.\]
Since
\[g=\begin{cases}h_1h_2\dots h_{\frac{N(g)-1}{2}}g_{N(g)} & \text{ if } N(g) \text{ is odd} \\
h_1h_2\dots h_{\frac{N(g)}{2}} & \text{ if } N(g) \text{ is even}, \end{cases}\]
we have
\[\sum_{j\in X^{l_\nu(\frac{6}{\epsilon})}} |g_j|_{\nu + l_\nu(\frac{6}{\epsilon})} \leq
\left(\sum_{i=1}^{\frac{N(g)-1}{2}}\sum_{j\in X^{l_\nu(\frac{6}{\epsilon})}}|h_{i,j}|_{\nu + l_\nu(\frac{6}{\epsilon})}\right) + \sum_{j\in X^{l_\nu(\frac{6}{\epsilon})}}|g_{N(g),j}|_{\nu + l_\nu(\frac{6}{\epsilon})}\]
if $N(g)$ is odd and
\[\sum_{j\in X^{l_\nu(\frac{6}{\epsilon})}} |g_j|_{\nu + l_\nu(\frac{6}{\epsilon})} \leq
\sum_{i=1}^{\frac{N(g)}{2}}\sum_{j\in X^{l_\nu(\frac{6}{\epsilon})}}|h_{i,j}|_{\nu + l_\nu(\frac{6}{\epsilon})}\]
if $N(g)$ is even, where $X^{l_\nu(\frac{6}{\epsilon})}$ is the set of words of length $l_\nu(\frac{6}{\epsilon})$ in the alphabet $\{1,2,\dots, d\}$,
\[h_i = (h_{i,11\dots 1}, h_{i, 11\dots 2}, \dots, h_{i, dd\dots d})\tau_i\]
is the $l_\nu(\frac{6}{\epsilon})\textsuperscript{th}$-level decomposition of $h_i$ and
\[g_{N(g)} = (g_{N(g),11\dots 1}, g_{N(g), 11\dots 2}, \dots, g_{N(g), dd\dots d})\tau_{N(g)}\]
is the $l_\nu(\frac{6}{\epsilon})\textsuperscript{th}$-level decomposition of $g_{N(g)}$.

It follows from the definition of $l_\nu(\frac{6}{\epsilon})$ that $h_i\notin \II_{l_\nu(\frac{6}{\epsilon})}^\nu$ for all $i\in S(g)$. Hence, for all $i\in S(g)$,
\[\sum_{j\in X^{l_\nu(\frac{6}{\epsilon})}}|h_{i,j}|_{\nu + l_\nu(\frac{6}{\epsilon})} \leq |h_i|_{\nu} - 1.\]
Therefore, as long as $n>\frac{3}{\epsilon}$,
\begin{align*}
\sum_{j\in X^{l_\nu(\frac{6}{\epsilon})}}|g_j|_{\nu+l_\nu(\frac{6}{\epsilon})} &\leq n-|S(g)| \\
&<n-\frac{\epsilon}{8}n \\
&=\frac{8-\epsilon}{8}n.
\end{align*}
It follows that for $n>\frac{3}{\epsilon}$,
\begin{align*}
|\Omega_{\nu}^{>}(n,\epsilon)| &\leq \sum_{k_1+\dots +k_{d^{l_\nu(\frac{6}{\epsilon})}} \leq \frac{8-\epsilon}{8}n}C|\Omega_{\nu + l_\nu(\frac{6}{\epsilon})}(k_1)|\dots |\Omega_{l_\nu(\frac{6}{\epsilon})}(k_{d^{l_\nu(\frac{6}{\epsilon})}})| \\
&\leq \left(\frac{8-\epsilon}{8}n\right)^{d^{l_\nu(\frac{6}{\epsilon})}}K(n)\kappa_{\nu + l_\nu(\frac{6}{\epsilon})}^{\frac{8-\epsilon}{8}n}
\end{align*}
where $C=\left[G_\nu : \St_{G\nu}\left(l_\nu\left(\frac{6}{\epsilon}\right)\right)\right]$ and $K(n)$ is a function such that
\[\lim_{n\to\infty}K(n)^{\frac{1}{n}} = 1.\]

We conclude that, for a fixed $\epsilon$ between $0$ and $1$,
\[\limsup_{n\to\infty}|\Omega_\nu^{>}(n, \epsilon)|^{\frac{1}{n}} \leq \kappa_{\nu + l_\nu(\frac{6}{\epsilon})}^{\frac{8-\epsilon}{8}}.\]

On the other hand,
\begin{align*}
|\Omega_\nu^{<}(n,\epsilon)| &\leq \sum_{i=1}^{\epsilon n}\sum_{k_1+\dots + k_i = n}\prod_{j=1}^{i}\delta(k_j)\\
&\leq \sum_{i=1}^{\epsilon n}\sum_{k_1+\dots + k_i = n}\delta\left(\frac{n}{i}\right)^{i}
\end{align*}
by lemma 6 of \cite{BartholdiPochon09}, since $\ln(\delta)$ is concave. Hence, assuming that $\epsilon < \frac{1}{2}$, we have
\begin{align*}
|\Omega_\nu^{<}(n,\epsilon)| &\leq \sum_{i=1}^{\epsilon n}\binom{n}{i-1}\max_{1\leq i \leq \epsilon n}\left\{\delta\left(\frac{n}{i}\right)^i\right\} \\
&\leq \epsilon n \binom{n}{\epsilon n}\max_{1\leq i \leq \epsilon n}\left\{\delta\left(\frac{n}{i}\right)^i\right\}.
\end{align*}
Using the fact that $\binom{n}{\epsilon n} \leq \frac{n^{\epsilon n}}{(\epsilon n )!}$ and Stirling's approximation, we get
\[|\Omega_\nu^{<}(n,\epsilon)| \leq \epsilon n \left(\frac{e}{\epsilon}\right)^{\epsilon n}\frac{C(n)}{\sqrt{2\pi\epsilon n}}\max_{1\leq i \leq \epsilon n}\left\{\delta\left(\frac{n}{i}\right)^i\right\}\]
where $\lim_{n\to\infty}C(n) = 1$. Therefore,
\[\limsup_{n\to\infty}|\Omega_\nu^{<}(n,\epsilon)|^{\frac{1}{n}} \leq \left(\frac{e}{\epsilon}\right)^{\epsilon}\limsup_{n\to\infty}\delta\left(\frac{n}{i_n}\right)^{\frac{i_n}{n}}\]
where $1\leq i_n \leq \epsilon n$ maximises $\delta\left(\frac{n}{i}\right)^{i}$. Let $k_n=\frac{n}{i_n}$. Then, $\frac{1}{\epsilon}\leq k_n \leq n$. Since $\lim_{k\to\infty}\delta(k)^{\frac{1}{k}} = 1$, there must exist $N\in\NN$ such that $\sup_{\frac{1}{\epsilon}\leq k}\{\delta(k)\} = \sup_{\frac{1}{\epsilon}\leq k \leq N}\{\delta(k)\}$. Hence, there exists some $K_\epsilon \in \NN$ such that $K_\epsilon \geq \frac{1}{\epsilon}$ and $\limsup_{n\to\infty}\delta\left(\frac{n}{i_n}\right)^{\frac{i_n}{n}} = \delta(K_{\epsilon})^{\frac{1}{K_\epsilon}}$. We conclude that
\[\limsup_{n\to\infty}|\Omega_\nu^{<}(n,\epsilon)|\leq \left(\frac{e}{\epsilon}\right)^{\epsilon}\delta(K_\epsilon)^{\frac{1}{K_\epsilon}}\]
for some $K_\epsilon \geq \frac{1}{\epsilon}$.

Since, for any $0<\epsilon<\frac{1}{2}$, we have $|\Omega_{\nu}(n)| = |\Omega_\nu^{>}(n,\epsilon)| + |\Omega_\nu^{<}(n,\epsilon)|$,
\begin{align*}
\kappa_\nu &= \lim_{n\to\infty}|\Omega_{\nu}(n)|^{\frac{1}{n}} = \lim_{n\to\infty}\left(|\Omega_\nu^{>}(n,\epsilon)| + |\Omega_\nu^{<}(n,\epsilon)|\right)^{\frac{1}{n}} \\
&\leq \limsup_{n\to\infty}\left(2 \max\left\{|\Omega_\nu^{>}(n,\epsilon)|, |\Omega_\nu^{<}(n,\epsilon)|\right\}\right)^{\frac{1}{n}} \\
&= \max \left\{\limsup_{n\to\infty}|\Omega_\nu^{>}(n,\epsilon)|^{\frac{1}{n}}, \limsup_{n\to\infty}|\Omega_\nu^{<}(n,\epsilon)|^{\frac{1}{n}}\right\}\\
&\leq \max\left\{\kappa_{\nu + l_\nu(\frac{6}{\epsilon})}^{\frac{8-\epsilon}{8}}, e^{\epsilon}\left(\frac{1}{\epsilon}\right)^{\epsilon } \delta\left(K_\epsilon\right)^{\frac{1}{K_\epsilon}}\right\}.
\end{align*}
Let us now fix $0<\epsilon<\frac{1}{2}$. There must exist a $k\in \NN$ such that
\[\kappa_{\nu+k} \leq e^{\epsilon}\left(\frac{1}{\epsilon}\right)^{\epsilon } \delta\left(K_\epsilon\right)^{\frac{1}{K_\epsilon}}.\]
Indeed, otherwise we would have $\kappa_{\nu+i} > e^{\epsilon}\left(\frac{1}{\epsilon}\right)^{\epsilon } \delta\left(K_\epsilon\right)^{\frac{1}{K_\epsilon}}$ for all $i\in\NN$. In particular, this would imply that
\[\kappa_\nu \leq \kappa_{\nu + l_\nu(\frac{6}{\epsilon})}^{\frac{8-\epsilon}{8}}.\]
Let $\nu'\in\NN$ be such that $\nu' \geq \nu + l_\nu(\frac{6}{\epsilon})$ and $\II_\infty^{\nu'} \cap \Omega_{\nu'}(n) \leq \delta(n)$ for all $n\in \NN$ (such a $\nu'$ exist by hypothesis). Then, we would also have
\[\kappa_{\nu'} \leq \kappa_{\nu' + l_{\nu'}(\frac{6}{\epsilon})}^{\frac{8-\epsilon}{8}},\]
and so, using the fact that by Proposition \ref{prop:GrowthRateCannotDecrease}, $\kappa_{\nu + l_{\nu}(\frac{6}{\epsilon})} \leq \kappa_{\nu'}$, we would have
\[\kappa_\nu \leq \kappa_{\nu' + l_{\nu'}(\frac{6}{\epsilon})}^{\left(\frac{8-\epsilon}{8}\right)^2}.\]
By induction, we conclude that for any $m\in \NN^*$, there exists $k_m\in \NN$ such that
\[\kappa_\nu \leq \kappa_{\nu + k_m}^{\left(\frac{8-\epsilon}{8}\right)^{m}}.\]
Since $|S_i| \leq A$ for every $i\in \NN$, we have that $\kappa_i \leq A$ for every $i\in \NN$. Hence, we get that $\kappa_\nu \leq A^{\left(\frac{8-\epsilon}{8}\right)^{m}}$ for every $m\in \NN^*$, which implies that $\kappa_\nu = 1$. This contradicts the hypothesis that $\kappa_\nu > e^{\epsilon}\left(\frac{1}{\epsilon}\right)^{\epsilon } \delta\left(K_\epsilon\right)^{\frac{1}{K_\epsilon}}$.

Therefore, there must exist some $i\in \NN$ such that $\kappa_{\nu + i} \leq e^{\epsilon}\left(\frac{1}{\epsilon}\right)^{\epsilon } \delta\left(K_\epsilon\right)^{\frac{1}{K_\epsilon}}$. By Proposition \ref{prop:GrowthRateCannotDecrease}, we must have
\[\kappa_\nu \leq e^{\epsilon}\left(\frac{1}{\epsilon}\right)^{\epsilon } \delta\left(K_\epsilon\right)^{\frac{1}{K_\epsilon}}.\]
As the above inequality is valid for any $0<\epsilon<\frac{1}{2}$ and 
\[\lim_{\epsilon\to 0}e^{\epsilon}\left(\frac{1}{\epsilon}\right)^{\epsilon } \delta\left(K_\epsilon\right)^{\frac{1}{K_\epsilon}} = 1\]
we must have $\kappa_\nu = 1$, and so $G_\nu$ is of subexponential growth.
\end{proof}

\section{Growth of spinal groups}\label{sec:GrowthSpinalGroups}

Using the techniques developed by Grigorchuk in \cite{Grigorchuk85}, one can show that every spinal group acting on the binary rooted tree is of subexponential growth.

In this section, we will study the growth of some spinal groups acting on the 3-regular rooted tree $T_3$. We will be able to prove that the growth is subexponential in several new cases. In particular, our results will imply that all the groups in \v{S}uni\'{k}'s family acting on $T_3$ (Example \ref{ex:SunicGroups}) are of subexponential growth. While this was already known for torsion groups, this was previously unknown for groups with elements of infinite order, except for the case of the Fabrykowski-Gupta group.

Unfortunately, we were unable to obtain similar results for spinal groups acting on rooted trees of higher degrees, as the methods used here do not seem to have obvious generalizations in those settings.

\subsection{Growth of spinal groups acting on $T_3$}

Let $m\in \NN$, $\ZZ/3\ZZ \cong A = \langle (1 2 3) \rangle \subseteq \Sym(3)$ and $B=(\ZZ/3\ZZ)^m$. Let
\[\Omega = \left\{\left\{\omega_{ij}\right\}_{i\in\NN, 1\leq j \leq 2} \mid \omega_{i,1}\in \Epi(B,A), \omega_{i,2} = 1, \bigcap_{i\geq k}\ker(\omega_{ij}) = 1 \forall k\in \NN\right\}\]
be a set of sequences of homomorphisms of $B$ into $A$ and $\sigma\colon \Omega \rightarrow \Omega$ be the left-shift (see Section \ref{subsubsection:SpinalGroups}). For any $\omega\in \Omega$, let us define $A_\omega = A$. Using the notation of Section \ref{subsubsection:SpinalGroups}, we get spinal groups $G_\omega = \langle A, B_\omega \rangle$ acting on $T_3$ which naturally come equipped with a word pseudonorm $|\cdot|_\omega$ assigning length 0 to elements of $A$ and length 1 to elements of $B_\omega$.

\begin{notation}
In order to streamline the notation, we will drop the indices $\omega$ wherever it is convenient and rely on context to keep track of which group we are working in. We will also drop the second index in the sequences of $\Omega$ and write $\omega=\omega_0\omega_1\dots \in \Omega$, which is a minor abuse of notation.

The set of incompressible elements of $G_\omega$ will be denoted by $\II_\infty^\omega$, and we will write $\II_\infty^\omega(n)$ for the set of incompressible elements of length $n$.

We will write $a=(1 2 3) \in A$, and for any $b\in B_\omega$, we will write $b^{a^i} = a^iba^{-i}$ where $i \in \ZZ/3\ZZ$.
\end{notation}

\begin{remark}
As in the case of the binary tree, we have that for every $\omega\in \Omega$, the group $G_\omega$ is a quotient of $A\ast B_\omega$. Hence, every element of $g_\omega$ can be written as an alternating product of elements of $A$ and $B_\omega$.

It follows that every $g\in G_\omega$ of length $n$ can be written as
\[g=\beta_1^{a^{c_1}}\beta_{2}^{a^{c_2}}\dots \beta_n^{a^{c_n}}a^{s}\]
for some $s\in \ZZ/3\ZZ$, $\beta\colon \{1,2,\dots, n\}\rightarrow B_\omega$ and $c\colon \{1,2,\dots, n\} \rightarrow \ZZ/3\ZZ$ (where we use indices to denote the argument of the function in order to make the notation more readable).
\end{remark}

\begin{notation}
For any $n\in \NN$ at least 2 and $c\colon \{1,2,\dots, n\} \rightarrow \ZZ/3\ZZ$, we will denote by $\partial c\colon \{1,2,\dots n-1\} \rightarrow \ZZ/3\ZZ$ the \emph{discrete derivative} of $c$, that is,
\[\partial c(k) = c_{k+1}-c_k.\]
\end{notation}

\begin{lemma}\label{lemma:DerivativeOfSequenceOfConjugation}
Let $\omega\in \Omega$ and $g\in G_\omega$ with $|g|=n$. Writing 
\[g=\beta_1^{a^{c_1}}\beta_{2}^{a^{c_2}}\dots \beta_n^{a^{c_n}}a^{s}\]
for some $n\in \NN$, $s\in \ZZ/3\ZZ$, $\beta\colon \{1,2,\dots, n\}\rightarrow B_\omega$ and $c\colon \{1,2,\dots, n\} \rightarrow \ZZ/3\ZZ$, if $g\in \II_\infty^\omega(n)$, then there exists $m_c\in \{1,2,\dots, n\}$ such that
\[\partial c(k) = \begin{cases}
2 \text{ if } k< m_c \\
1 \text{ if } k\geq m_c.
\end{cases}\]
\end{lemma}
\begin{proof}
If there exists $k\in \{1,2,\dots, n-1\}$ such that $\partial c(k) = 0$, then $c(k)=c(k+1)$, which means that 
\[|g|=|\beta_1^{a^{c_1}}\beta_2^{a^{c_2}}\dots (\beta_k\beta_{k+1})^{a^{c_k}} \dots \beta_n^{a^{c_n}}a^{s}| \leq n-1\]
a contradiction. Hence, $\partial c(k) \ne 0$ for all $k\in \{1,2,\dots, n-1\}$.

Therefore, to conclude, we only need to show that if $\partial c(k) = 1$ for some $k\in\{1,2,\dots, n-2\}$, then $\partial c(k+1) \ne 2$. For the sake of contradiction, let us assume that $\partial c(k) = 1$ and $\partial c(k+1) = 2$ for some $k\in\{1,2,\dots, n-2\}$. Without loss of generality, we can assume that $c_k=0$ (indeed, it suffices to conjugate by the appropriate power of $a$ to recover the other cases). We have
\begin{align*}
\beta_k\beta_{k+1}^{a}\beta_{k+2} &= (\alpha_k, 1, \beta_k)(\beta_{k+1},\alpha_{k+1}, 1)(\alpha_{k+2}, 1, \beta_{k+2})\\
&=(\alpha_k\beta_{k+1}\alpha_{k+2}, \alpha_{k+1}, \beta_{k}\beta_{k+2})
\end{align*}
for some $\alpha_k,\alpha_{k+1}, \alpha_{k+2} \in A$. Since
\[|\alpha_k\beta_{k+1}\alpha_{k+2}| + | \alpha_{k+1}| + |\beta_{k}\beta_{k+2}| = 2 < 3 = |\beta_k\beta_{k+1}^{a}\beta_{k+2}|,\]
there is some length reduction on the first level, so $g\notin \II_\infty^\omega$.
\end{proof}

It follows from Lemma \ref{lemma:DerivativeOfSequenceOfConjugation} that an element
\[g=\beta_1^{a^{c_1}}\beta_{2}^{a^{c_2}}\dots \beta_n^{a^{c_n}}a^{s}\in \II_\infty^\omega\]
is uniquely determined by the data $(\beta, s, c_1, m_c)$, where $\beta\colon \{1,2,\dots, n\} \rightarrow B_\omega$, $s,c_1\in \ZZ/3\ZZ$ and $m_c\in \{1,2,\dots, n\}$. Of course, not every possible choice corresponds to an element of $\II_\infty^\omega$. In what follows, we will bound the number of good choices for $(\beta, s, c_1, m_c)$.

\begin{prop}\label{prop:T3PolynomialGrowthOfIncompressibles}
Let $\omega = \omega_0\omega_1\omega_2\dots \in \Omega$ and let $l\in \NN$ be the smallest integer such that $\cap_{i=0}^l\ker(\omega_i) = 1$. Then, there exists a constant $C_l\in \NN$ such that
\[|\II_\infty^\omega(n)| \leq C_ln^{\frac{3^{l+2} - 1}{2}}\]
for all $n\in \NN$.
\end{prop}
\begin{proof}
Let us fix $n\in \NN$, $s, c_1\in\ZZ/3\ZZ$ and $m_c\in \{1,2,\dots, n\}$, and let $c\colon \{1,2,\dots, n\}\rightarrow \ZZ/3\ZZ$ be the unique sequence such that $c(1)=c_1$ and
\[\partial c(k) = \begin{cases}
2 \text{ if } k< m_c \\
1 \text{ if } k\geq m_c
\end{cases}\]
for all $k\in\{1,2,\dots n-1\}$. We will try to bound the number of maps $\beta\colon \{1,2,\dots,n\}\rightarrow B_\omega\setminus\{1\}$ such that
\[g=\beta_1^{a^{c_1}}\beta_{2}^{a^{c_2}}\dots \beta_n^{a^{c_n}}a^{s} \in \II_\infty^\omega(n).\]
Assuming that $g\in \II_\infty^\omega(n)$, let us look at the first-level decomposition of $g$,
\[g=(g_1,g_2,g_3)a^s.\]
Since $g\in \II_\infty^\omega$, we must have $|g|=|g_1|+|g_2|+|g_3|$. As for any $k\in \{1,2,\dots,n\}$, $\beta_k^{a^{c_k}}$ adds 1 to the length of $g$ and to the length of exactly one of $g_1,g_2$ or $g_3$ no matter the value of $\beta_k$, we conclude that $|g_1|, |g_2|, |g_3|$ do not depend on $\beta$.

Since $g\in \II_\infty^\omega$, we must also have that $g_1,g_2,g_3\in\II_\infty^\omega$. Hence, for $i=1,2,3$, there must exist $s^{(i)},c_1^{(i)} \in \ZZ/3\ZZ$, $m_c^{(i)}\in \{1,2,\dots, |g_i|\}$ and $\beta^{(i)} \colon \{1,2,\dots, |g_i|\}\rightarrow B_{\sigma(\omega)}\setminus \{1\}$ such that
\[g_i=(\beta_1^{(i)})^{a^{c_1^{(i)}}}(\beta_{2}^{(i)})^{a^{c_2^{(i)}}}\dots (\beta_{|g_i|}^{(i)})^{a^{c_{|g_i|}^{(i)}}}a^{s^{(i)}}\]
where $c^{(i)}$ is the unique map satisfying $c^{(i)}(1)=c_1^{(i)}$ and
\[\partial c^{(i)}(k) = \begin{cases}
2 \text{ if } k< m_c^{(i)} \\
1 \text{ if } k\geq m_c^{(i)}.
\end{cases}\]
It is clear that the maps $\beta^{(i)}$ are completely determined by the map $\beta$. Therefore, to specify $g_1,g_2,g_3$, we only need to consider $s^{(i)}, c_1^{(i)}$ and $m_c^{(i)}$. However, the choice of $s^{(i)}, c_1^{(i)}$ and $m_c^{(i)}$ impose some non-trivial conditions on $\beta$. Indeed, once these three numbers are fixed, we have
\[g_i=a^{k_1^{(i)}}\square a^{k_2^{(i)}}\square \dots \square a^{k_{|g_i|}^{(i)}}\square a^{k_{|g_i|+1}^{(i)}} \]
where $\square$ are unspecified elements of $B_{\sigma(\omega)}\setminus \{1\}$ and the $k_j$ are uniquely determined by $s^{(i)}, c_1^{(i)}$ and $m_c^{(i)}$. These $k_j^{(i)}$ completely determine $\omega_0(\beta_k)$ for all but at most one $k\in \{1,2,\dots, n\}$. Indeed, for $k\in\{1,2,\dots,n\}$, we have
\[\beta_k^{a^{c_k}} = 
\begin{cases}
(\alpha_k^{\omega_0}, 1, \beta_k) & \text{ if } c_k=0 \\
(\beta_k, \alpha_k^{\omega_0}, 1) & \text{ if } c_k=1 \\
(1, \beta_k, \alpha_k^{\omega_0}) & \text{ if } c_k=2
\end{cases}\]
where $\alpha_k^{\omega_0} = \omega_0(\beta_k)\in A$. As long as $\partial c$ is constant, $c$ is a subsequence of
\[\dots 021021021 \dots\]
or
\[\dots 012012012 \dots \]
which means that $g_1,g_2$ and $g_3$ will be given by an alternating product of $\alpha_k^{\omega_0}$ and $\beta_k$, except perhaps at $m_c$, if $1<m_c<n$. Let us assume for the sake of illustration that $c_{m_c-1} = 0$ (the other two cases are obtained simply by permuting the indices). In that case, we have
\begin{align*}
\beta_{m_c-1}\beta_{m_c}^{a^{2}}\beta_{m_c+1} &= (\alpha_{m_c-1}^{\omega_0}, 1, \beta_{m_c-1})(1,\beta_{m_c}, \alpha_{m_c}^{\omega_0})(\alpha_{m_c+1}^{\omega_0}, 1, \beta_{m_c+1})\\
&=(\alpha_{m_c - 1}^{\omega_0}\alpha_{m_c+1}^{\omega_0}, \beta_{m_c}, \beta_{m_c-1}\alpha_{m_c}^{\omega_0}\beta_{m_c+1})
\end{align*}
Therefore, the choice of $s^{(i)}, c_1^{(i)}$ and $m_c^{(i)}$ give us conditions on $\omega_0(\beta_k)$ for all but at most one $k$.

By induction, on level $l+1$, the choice of $s^{(x)}, c_1^{(x)},m_c^{(x)}$ for all words $x$ of length at most $l+1$ in the alphabet $\{1,2,3\}$ (that is, for all vertices of the tree up to level $l+1$) determine $\omega_i(\beta_k)$ for all $1\leq i \leq l$ and for all $k\in\{1,2,\dots, n\}$ except for at most $\sum_{j=0}^{l}3^{j} = \frac{3^{l+1}-1}{2}$. Since $\cap_{i=0}^l\ker{\omega_i} = \{1\}$, for each $k$, there is at most one $\beta_k\in B$ having the prescribed images $\omega_i(\beta_k)$ for all $1\leq i \leq l$.

Since there are $\frac{3^{l+2}-1}{2}$ vertices in the tree up to level $l+1$, we have $3^{\frac{3^{l+2}-1}{2}}$ choices for $s(x)$ and $3^{\frac{3^{l+2}-1}{2}}$ choices for $c_1(x)$. Since $m_c^{(x)}$ satisfies $1\leq m_c^{(x)} \leq n$, there are at most $n^{\frac{3^{l+2}-1}{2}}$ choices for $m_c^{(x)}$. Once all these choices are made, $\beta$ is completely determined, except for at most $\frac{3^{l+1}-1}{2}$ values. For each of these, we have $|B|-1$ choices, so there are at most $(|B|-1)^{\frac{3^{l+1}-1}{2}}$ choices for $\beta$. Hence, there are at most
\[C_ln^{\frac{3^{l+2}-1}{2}}\]
elements in $\II_\infty^\omega(n)$, where
\[C_l = 3^{3^{l+2}-1}(|B|-1)^{\frac{3^{l+1}-1}{2}}.\]

\end{proof}

With this, we can prove that many spinal groups are of subexponential growth.

\begin{thm}
Let $\omega\in\Omega$ and $G_\omega$ be the associated spinal group of automorphisms of $T_3$. If there exists $l\in \NN$ such that $\cap_{i=k}^{k+l}\ker(\omega_i) = 1$ for infinitely many $k\in \NN$, then $G_\omega$ is of subexponential growth.
\end{thm}
\begin{proof}
According to Proposition \ref{prop:T3PolynomialGrowthOfIncompressibles}, there exist infinitely many $k\in \NN$ such that
\[|\II_\infty^{\sigma^{k}(\omega)}(n)| \leq C_ln^{\frac{3^{l+2}-1}{2}}\]
for some $C_l\in\NN$. Since $\ln(C_ln^{\frac{3^{l+2}-1}{2}})$ is concave, the result follows from Theorem \ref{thm:GrowthCriterion}.
\end{proof}

\section*{Acknowledgements}
This work was supported by the Natural Sciences and Engineering Research Council of Canada. The author would like to thank Tatiana Nagnibeda for many useful discussions and suggestions, as well as Laurent Bartholdi and Rostislav Grigorchuk for reading previous versions of this paper and offering helpful comments.

\bibliography{growth}
\bibliographystyle{plain}

\end{document}